\documentclass[11pt,a4paper]{article}

\usepackage[T1]{fontenc}
\usepackage{amsmath,amsthm,amsopn,amsfonts,amssymb,dsfont,color}
\usepackage{mathrsfs}
\usepackage{statex}
\usepackage{graphicx}
\usepackage{float}
\usepackage{xcolor}
\usepackage{mathrsfs}
\usepackage{mathtools}
\usepackage{accents}
\usepackage{setspace}
\usepackage{times}
\usepackage{listings}
\usepackage{enumerate}
\usepackage{indentfirst}
\usepackage{a4,a4wide}
\usepackage{tikz}

\newtheorem{theorem}{\textbf{Theorem}}[section]
\newtheorem{lemma}[theorem]{\textbf{Lemma}}
\newtheorem{proposition}[theorem]{\textbf{Proposition}}
\newtheorem{claim}[theorem]{\textbf{Claim}}

\newtheorem{remark}[theorem]{\textbf{Remark}}

\numberwithin{equation}{section}
\numberwithin{figure}{section}

\makeatletter
\g@addto@macro\th@plain{\thm@headpunct{}}
\makeatother

\newcommand\bea{\begin{eqnarray}}
\newcommand\eea{\end{eqnarray}}
\newcommand\beaa{\begin{eqnarray*}}
\newcommand\eeaa{\end{eqnarray*}}

\title{Sufficient conditions for determining the sign of the wave speed in the Lotka-Volterra competition
system}
\date{}
\author{}
\begin{document}

\maketitle
\vspace{-30pt}
\begin{center}
{\large\bf

Dongyuan Xiao\footnote{Graduate School of Mathematical Science, The University of Tokyo, Tokyo, Japan.

e-mail: {\tt dongyuanx@hotmail.com}}

}
\end{center}

\begin{abstract}
This paper mainly focuses on the sign of the wave speed in the Lotka-Volterra competition system of bistable type, also known as the strong-strong competition case. The traveling wave solution of the system is crucial for understanding the long-time behavior of solutions to the Cauchy problem. Specifically, the sign of the wave speed is key to predicting which species will prevail in the competition. In this paper, by studying a degenerate Lotka-Volterra competition system, we propose two sufficient conditions for determining the sign of the wave speed.
\\


\end{abstract}
\section{Introduction}

The main target of this paper is the two-species Lotka-Volterra competition system, which is of significant biological relevance \cite{Murray1993}:
\begin{equation}\label{system}
\left\{
\begin{aligned}
&u_t=u_{xx}+u(1-u-av), & t>0,\ x\in\mathbb{R},\\
&v_t=dv_{xx}+rv(1-v-bu), & t>0,\ x\in \mathbb{R}.
\end{aligned}
\right.
\end{equation}
In this system, $u=u(t,x)$ and $v=v(t,x)$ represent the population densities of two competing species at the time $t$ and position $x$; $d$ and $r$ stand for the diffusion rate and intrinsic growth rate of $v$, respectively; $a$ and $b$ represent the competition coefficient of $v$ and $u$, respectively. Here, all parameters are assumed to be positive and satisfy the bistable structure, {\it i.e.},
$a$ and $b$ satisfy
\begin{itemize}
\item[{\bf(H1)}] $a>1$ and $b>1$.
\end{itemize}

Regarding the traveling wave solution of \eqref{system} under the assumption {\bf(H1)},
Kan-on  \cite{Kan-On} showed
that there exists a unique traveling wave speed $\hat c\in[-2\sqrt{dr},2]$
such that \eqref{system} admits a positive solution
$(u,v)(x,t)=(U,V)(x- ct)$ with $c=\hat c$
 satisfying
\begin{equation}\label{tw solution}
\left\{
\begin{aligned}
&U''+cU'+U(1-U-aV)=0,\\
&dV''+cV'+rV(1-V-bU)=0,\\
&(U,V)(-\infty)=(1,0),\ (U,V)(\infty)=(0,1),\\
&U'<0,\ V'>0.
\end{aligned}
\right.
\end{equation}

The sign of the wave speed $\hat c$ is crucial in determining which species is dominant and, consequently, the long-time behavior of the solution to \eqref{system}. Under the assumption that $\hat c > 0$, the propagation phenomenon was first proven by Carrère \cite{Carrere}. More recently, Peng, Wu, and Zhou \cite{Peng Wu Zhou} provided refined estimates for both the spreading speed and the solution's profile.

Although the sign of speed has been studied widely in the literature,
it is still not completely understood. So far, the sign of $\hat c$ can be determined only for some parameter regions. For example, Rodrigo and Mimura \cite{Rodrigo Mimura} constructed some exact solutions to determine the sign of $\hat c$; 
Guo and Lin \cite{Guo Lin} used some fundamental analysis to provide some explicit conditions that can determine the sign directly. In particular, their results conclude that
\begin{itemize}
    \item When $r = d$, then $\hat c>0$ if $b > a > 1$, $\hat c=0$ if $a = b > 1$,  and $\hat c<0$ if $a > b > 1$;
    \item When $r > d$, then $\hat c>0$ if $a > 1$ and $b\ge \frac{r^2a}{d^2}$;
\item When $r < d$, then $\hat c<0$ if $b > 1$ and $a\ge \frac{d^2b}{r^2}$.
 \end{itemize}
Later, Girardin and Nadin \cite{Girardin Nadin} considered both $a$ and $b$ to be sufficiently large and used the singular limit approach to study the relation between the sign of $\hat c$ and the diffusion speed $d$. Ma, Huang, and Ou \cite{Ma Huang Ou}, and Morita, Nakamura, and Ogiwara \cite{Morita_etal2023} applied the upper-lower solutions method to establish criteria for the sign of $\hat c$. Chang, Chen, and Wang \cite{ChangChenWang2023} proposed a mini-max characterization to derive the speed sign.
We also refer to Girardin \cite{Girardin} for a survey on this topic. 

In this paper, we always fix $b>1$, $r,d>0$, while setting the competition rate $a>1$ as a continuously 
varying parameter. As a result, by the simple comparison argument,  we can view the traveling wave speed $\hat c(a)$ as a nonincreasing continuous function of $a$. Our main result follows from a straightforward observation.
Since $a$ indicates the competition strength of $v$ species, when $a$ is sufficiently large, $v$ naturally becomes the stronger species and wins the competition. Conversely,  when $a$ is sufficiently close to $1$, $u$ naturally becomes the stronger species and wins the competition.
\begin{theorem}\label{th:conditions}
Let $r,d>0$ and $b>1$ in \eqref{system} be fixed. Then there exist $a_1>a_2>1$ such that for all $a\ge a_1$ the traveling wave speed satisfies $\hat c<0$, and for all $a\in (1,a_2]$ the traveling wave speed satisfies $\hat c>0$.
\end{theorem}

The existence of $a_1$ can be proved by taking the limit as $a \to \infty$ and applying the singular limit approach, whereas the existence of $a_2$ is more complex. To address this, we  study a degenerate Lotka-Volterra system 
\begin{equation*}
\left\{
\begin{aligned}
&u_t=u_{xx}+u(1-u-v), & t>0,\ x\in\mathbb{R},\\
&v_t=dv_{xx}+rv(1-v-bu), & t>0,\ x\in \mathbb{R},
\end{aligned}
\right.
\end{equation*}
which can also be viewed as a threshold between the monostable and bistable types of the Lotka-Volterra system.

\bigskip
Next, we briefly introduce the Lotka-Volterra system of monostable type, commonly referred to as the strong-weak competition case. We assume the following condition:
\begin{itemize}
\item[{\bf(H2)}] $0<a<1$ and $b>1$.
\end{itemize}
The corresponding ODE system of \eqref{system} admits a unique stable state $(u,v)=(1,0)$, indicating that $u$ is the strong species and $v$ is the weak species.  Unlike the bistable case, for the monostable case, there exists a minimal traveling wave speed $c^*\ge 2\sqrt{1-a}$ such that \eqref{system} admits traveling wave solution $(u,v)(t,x)=(U,V)(x-ct)$ satisfying \eqref{tw solution} if and only if $c\ge c^*$. 

One of the most significant open problems for the Lotka-Volterra competition system is understanding the mechanism that determines whether  $c^*=2\sqrt{1-a}$ or $c^*>2\sqrt{1-a}$.
The linear and nonlinear selection of $c^{*}$ can be defined as follows:
\begin{itemize}
\item
It is linearly selected if
$c^{*}=2\sqrt{1-a}$ since the linearization of  \eqref{system}
at the unstable state $(u,v)=(0,1)$ results in the linear speed $2\sqrt{1-a}$.
In this case, the spreading speed is determined only by the leading edge of the population distribution.
\item 
when $c^{*}>2\sqrt{1-a}$, we say that the minimal traveling wave speed $c^{*}$ is nonlinearly selected.
In this case, the spreading speed is determined not only by the behavior of the leading edge of the population distribution but by the entire wave.
\end{itemize}
We also refer to the work of Roques et al. \cite{Roques et al 2015} that introduced another definition of the pulled front and the pushed front for \eqref{system}.


Under the assumption {\bf{(H2)}}, sufficient conditions for linear or nonlinear selection mechanism for \eqref{system} have been investigated
widely.  Okubo et al. \cite{Okubo etal} used a heuristic argument to conjecture that the minimal speed $c^{*}$ is linearly selected.
Later, Hosono \cite{Hosono 1998} suggested that $c^{*}$ can be nonlinearly selected in some parameter regimes.
The first explicit condition was proposed by Lewis, Li, and Weinberger \cite{Lewis Li Weinberger 1}, who stated  that linear selection holds when
\beaa
0<d<2\quad {\rm and}\quad  r(ab-1)\leq (2-d)(1-a).
\eeaa
This sufficient condition for linear selection was later improved by Huang \cite{Huang2010}, who provided the following condition:
\begin{equation*}
\frac{(2-d)(1-a)+r}{rb}\ge \max\Big\{a,\frac{d-2}{2|d-1|}\Big\}.
\end{equation*}
Roques et al. \cite{Roques et al 2015} numerically reported that the region of the parameter for linear determinacy could still be further improved. 
For the minimal speed $c^{*}$ to be nonlinearly selected, Huang and Han \cite{HuangHan2011} constructed examples in which linear determinacy fails to hold.
Holzer and Scheel \cite{Holzer Scheel 2012} showed that, for fixed $a$, $b$, and $r$,
 the minimal speed $c^{*}$ becomes nonlinear selection as $d\to\infty$. 
For related discussions, we also refer to, e.g.,
\cite{Alhasanat Ou2019-1, Alhasanat Ou2019, Guo Liang, Hosono 1995, Hosono 2003} and the references cited therein.

\bigskip
Let $r,d>0$ and $b>1$ be fixed. Regardless of whether we consider the bistable case $a>1$, the monostable case $0<a<1$, or the degenerate case $a=1$, we can define the asymptotic speed of spread $c^{**}$ (in short, spreading speed) of the solution starting from the initial datum 
\begin{equation}\label{initial datum}
u_0(x)\ge 0\ \ \text{compactly supported continuous function,}\ v_0(x)>0\ \ \text{uniformly positive},
\end{equation}
in the sense as follows:
\begin{equation*}
\left\{
\begin{aligned}
&\lim_{t\to\infty}\sup_{|x|\ge ct}\Big(u(t,x)+|1-v(t,x)|\Big)=0\ \ \mbox{for all} \ \ c>c^{**};\\
&\lim_{t\to\infty}\sup_{|x|\le ct}\Big(|1-u(t,x)|+v(t,x)\Big)=0\ \ \mbox{for all} \ \ c<c^{**}.
\end{aligned}
\right.
\end{equation*}
To emphasize the dependence on $a$, we denote the spreading speed by $c^{**}(a)$. By proving Theorem \ref{th: continuity} below, we can not only conclude the existence of $a_2$ but also establish a new sufficient condition to determine the speed of nonlinear selection. 
\begin{theorem}\label{th: continuity}
Let $r,d>0$ and $b>1$ in \eqref{system} be fixed. The spreading speed $c^{**}(a)$ is continuous for all $a\in\mathbb{R}^+$. Moreover, $c^{**}(1)$ is strictly positive.
\end{theorem}
More precisely, by using the continuity argument, it follows immediately from Theorem \ref{th: continuity} that the spreading speed becomes nonlinear selection as $a\to 1$ from below. This generalizes the result given in \cite{HuangHan2011}.
\begin{theorem}\label{th:speed selection}
Let $r,d>0$ and $b>1$ in \eqref{system} be fixed. Then there exists $\varepsilon>0$ such that for all $a\in[1-\varepsilon,1]$ the minimal traveling wave speed satisfies $c^*(a)=c^{**}(a)>2\sqrt{1-a}$.
\end{theorem}

\section{Preliminary}

\subsection{Comparison principle for the Lotka-Volterra competition system}

\noindent

Let us start by briefly recalling the competitive comparison principle.
Consider a domain $
\Omega:=(t_1,t_2)\times(x_1,x_2)$ with  $0\le t_1<t_2\le + \infty$ and $-\infty\le x_1<x_2\le+\infty$. A (classical) super-solution is a pair  $(\overline{u},\underline{v})\in \Big[C^1\Big((t_1,t_2),C^2((x_1,x_2))\Big)\cap C_b\left(\overline \Omega\right)\Big]^2$
satisfying
\begin{equation*}
\overline u_t-\overline u_{xx}-\overline u(1-\overline u-a\underline v)\ge 0\quad \text{ and } \quad \overline v_t-d\overline v_{xx}-r\overline v(1-\overline v-b\overline u)\le 0\ \ \text{in}\ \ \Omega.
\end{equation*}
Similarly, a (classical) sub-solution $(\underline u, \overline v)$ requires  
\begin{equation*}
\overline u_t-\overline u_{xx}-\overline u(1-\overline u-a\overline v)\le 0\quad \text{ and } \quad \overline v_t-d\overline v_{xx}-r\overline v(1-\overline v-b\underline u)\ge 0\ \ \text{in}\ \ \Omega.
\end{equation*}

\begin{proposition}[Comparison Principle]
Let $(\overline{u},\underline{v})$ and $(\underline{u},\overline{v})$ be a super-solution and sub-solution of system (\ref{system}) in $\Omega$, respectively. If
$$
\overline{u}(t_1,x)\ge \underline{u}(t_1,x) \quad \text{and} \quad \underline{v}(t_1,x)\le\overline{v}(t_1,x),\quad\text{for all } x\in (x_1,x_2),
$$
and, for $i=1,2$, 
$$
\overline{u}(t,x_i)\ge \underline{u}(t,x_i) \quad \text{and} \quad \underline{v}(t,x_i)\le\overline{v}(t,x_i),\quad\text{for all } t\in(t_1,t_2),
$$
then, it holds
$$ 
\overline{u}(t,x) \ge\underline{u}(t,x) \quad \text{ and } \quad \underline{v}(t,x)\le\overline{v}(t,x),\quad  \text{for all } (t,x)\in\Omega.
$$
If $x_1=-\infty$ or $x_2=+\infty$, the hypothesis on the corresponding boundary condition can be omitted.
\end{proposition}

We refer to the clear exposition of  {\it generalized} sub- and super-solutions in \cite[\S 2.1]{Girardin Lam} for more details. In particular,  if $(\underline {u}_1,\overline{v})$  and $(\underline {u}_2, \overline{v})$ are both classical sub-solutions, then $(\max(\underline {u}_1,\underline{u}_2),\overline v)$ is a generalized sub-solution. Also, 
if $(\underline u,\overline{v}_1)$    and $(\underline u, \overline{v}_2)$ are both classical sub-solutions, then $(\underline u,\min(\overline{v}_1,\overline{v}_2))$ is a generalized sub-solution.

\subsection{Existence of traveling waves for the case $a=1$}
We first show the existence of traveling waves for the degenerate case $a=1$ and $b> 1$ by the approximation argument. 
By simple comparison argument, $c^*(a)$ is nonincreasing for $a\in(0,1)$.
Therefore,
\bea\label{c-}
c_-:=\lim_{a\to 1^-} c^*(a)\in[0,2)
\eea
is well-defined.

\begin{proposition}\label{prop:existence-tw}
    Let $r,d>0$ and $b>1$ be given.  For any 
        $c\geq c_-$, 
     there exists a monotone solution $(c,U,V)$ satisfies \eqref{tw solution} with $a=1$. 
\end{proposition}
\begin{proof}
Let $\{c_k\}\to c_-$ as $k\to \infty$. First, note that $c^*(a)$ is nonincreasing for $a\in(0,1)$. Hence,
for each $c_k>c_-\ge 0$, there exists a sequence $(c_k, U_n, V_n)$ such that $(U_n, V_n)$ is a monotone solution to \eqref{tw solution} with $a=a_n\in(0,1)$ and speed $c_k$.
First, by translation, we may assume that $V_n(0)=1/2$ for all $n$. Also, by transferring the equation into integral equations (using a variation of the constants formula),
it is not hard to see that $U'_n$ and  $V'_n$ are uniformly bounded. Together with the fact that
$0\leq U_n(\xi), V_n(\xi)\leq 1$ for all $\xi\in\mathbb{R}$
and $n\in\mathbb{N}$, Arzel\`{a}-Ascoli Theorem allows us to take a subsequence that converges to a pair of limit functions $(U,V)\in [C(\mathbb{R})]^2$ with
$0\leq U,V\leq 1$, locally uniformly in $\mathbb{R}$. Moreover,
using Lebesgue's dominated convergence theorem to integral equations, we can conclude that
$(c_k,U,V)$ satisfies \eqref{tw solution} with $a=1$ (since $a_n\nearrow 1$).
Moreover, we can see from the equations satisfied by $U$ and $V$ that $(U,V)\in [C^2(\mathbb{R})]^2$ and $U'\leq 0$ and $V'\geq0$ (since $U'_n\leq 0$ and $V'_n\geq0$ for all $n$), which implies that $(U,V)(\pm\infty)$ exists.

It remains to show that
\bea\label{BC-cond-Appendix}
(U,V)(-\infty)=(1,0),\quad  (U,V)(+\infty)=(0,1).
\eea
 Note that we must have
\bea\label{bdry-cond}
U(\pm\infty)[1-U(\pm\infty)-aV(\pm\infty)]=0,\quad V(\pm\infty)[1-V(\pm\infty)-U(\pm\infty)]=0.
\eea
Hence, $U(\pm\infty)$, $V(\pm\infty)\in\{0,1\}$. Since $V_n(0)=1/2$ for all $n$, we have $V(0)=1/2$ and thus
\bea\label{lim-of-V}
V(-\infty)=0,\quad V(+\infty)=1.
\eea
Also, note that from \eqref{bdry-cond} we see that $V(+\infty)=1$ implies that
\bea\label{lim-of-U}
U(+\infty)=0.
\eea

If $U(-\infty)=0$, then $U\equiv0$ due to $U'\leq 0$. However, by integrating the equation of $V$ over $(-\infty,+\infty)$,
it follows that
\beaa
c_k+r\int_{-\infty}^{\infty}V(\xi)(1-V(\xi))d\xi=0,
\eeaa
which implies that $c_k<0$. This contradicts with $c_k\ge 0$
(more precisely, from \cite{Kan-on1997} we see that $2\sqrt{1-a}\leq c_k\leq 2$).
As a result, we have $U(-\infty)=1$, which together with \eqref{lim-of-V} and \eqref{lim-of-U} implies
\eqref{BC-cond-Appendix}.
We, therefore, obtain a monotone solution with $a=1$ and $c=c_k$.

Furthermore, by standard approximation argument, we can fix $a=1$ and take $c_k\downarrow c_-$
such that the system \eqref{tw solution} also admits a monotone solution with $a=1$ and speed $c=c_-$. This completes the proof.
\end{proof}

\subsection{Existence of the minimal traveling wave speed for the case $a=1$}
We next show that the minimal wave speed exists for $a=1$ and $b>1$.
To prove this, we first show 
that any traveling wave speed for $a = 1$ and $b > 1$ must be strictly positive.

\begin{proposition}\label{prop: minimal tw speed c>0}
Let $(c,U,V)$ satisfy \eqref{tw solution} with $a=1$ and $b>1$. Then $c>0$.
\end{proposition}
\begin{proof} 
We assume $c\le 0$, then get the contradiction by some basic properties of traveling wave solutions. We first consider the case $d\le 1$.
We rewrite \eqref{tw solution} as 
\begin{equation*}
        \left\{
            \begin{aligned}
                &U''+cU'+U(1-U-V)=0,\\
                &V''+\frac{c}{d}V'+\frac{r}{d}V(1-V-bU)=0,\\
                &(U,V)(-\infty)=(1,0),\quad (U,V)(+\infty)=(0,1),\\
                &U'<0,\quad V'>0,\quad \xi\in\mathbb{R},
            \end{aligned}
            \right.
    \end{equation*}
and add the both sides of $U$-equation and $V$-equation to get
\begin{equation}\label{eq of U+V}
(U+V)''+c(U+V)'+c(\frac{1}{d}-1)V'+(U+\frac{r}{d}V)(1-U-V)+\frac{r}{d}(1-b)UV=0.
\end{equation}
We claim that 
\bea\label{U+V<1}
U(\xi)+V(\xi)<1 \quad \mbox{for all $-\infty<\xi<+\infty$}
\eea
If not, from the boundary conditions, we can find a local maximum point $\xi_0$ such that 
$$(U+V)(\xi_0)\geq 1,\ (U+V)'(\xi_0)=0,\ \ \text{and}\ \ (U+V)''(\xi_0)\leq0.$$
Then, since $c\leq 0$, $d\le 1$, $V'>0$, and $b>1$, the left hand of \eqref{eq of U+V} at $\xi=\xi_0$ is strictly negative. Therefore, we get the contradiction, and \eqref{U+V<1} holds.


To deal with the case $d> 1$, we rewrite \eqref{tw solution} as 
\begin{equation*}
        \left\{
            \begin{aligned}
                &bU''+cbU'+bU(1-bU-V)+b(b-1)U^2=0,\\
                &V''+\frac{c}{d}V'+\frac{r}{d}V(1-V-bU)=0,\\
                 &(U,V)(-\infty)=(1,0),\quad (U,V)(+\infty)=(0,1),\\
                &U'<0,\quad V'>0,\quad \xi\in\mathbb{R},
            \end{aligned}
            \right.
    \end{equation*}
and add the both sides of $U$-equation and $V$-equation to get
\begin{equation}\label{eq of bU+V}
(bU+V)''+c(bU+V)'+c(\frac{1}{d}-1)V'+(bU+\frac{r}{d}V)(1-bU-V)+b(1-b)U^2=0.
\end{equation}
Next, we will show 
\bea\label{bU+V>1}
bU(\xi)+V(\xi)> 1 \quad \mbox{for all $-\infty<\xi<+\infty$}
\eea
If not, we can find a local minimum point $\xi_0$ such that 
$$(bU+V)(\xi_0)\leq 1,\ (bU+V)'(\xi_0)=0,\ \ \text{and}\ \ (bU+V)''(\xi_0)\geq 0.$$
Then, since $c\leq 0$, $d> 1$, $V'>0$, and $b>1$, the left hand of \eqref{eq of bU+V} at $\xi=\xi_0$ is strictly positive. Therefore, we get the contradiction, and \eqref{bU+V>1} holds. This completes the proof.
\end{proof}

Define the {\it admissible set of traveling wave speed} as
    \bea\label{A}
    \mathcal{A}:=\{\tilde{c}\in\mathbb{R}|\, \mbox{\eqref{tw solution} with $a=1$ has a monotone solution with $c=\tilde{c}$}\}.
        \eea
In view of Proposition~\ref{prop:existence-tw} 
and Proposition~\ref{prop: minimal tw speed c>0}, we have

\begin{proposition}\label{cor:c_lv(1)}
Let $r,d>0$ and $b>1$ be given. Then the minimal speed of traveling wave solutions
of \eqref{tw solution} with $a=1$ exists and is positive. That is, 
\beaa
c^{*}(1):=\min \mathcal{A}\in(0, c_-]
\eeaa
is well-defined.
where $\mathcal A$ is defined in \eqref{A}.
\end{proposition}
\begin{proof}
By Proposition~\ref{prop:existence-tw} and Proposition~\ref{prop: minimal tw speed c>0},
we see that the set $\mathcal A$ is nonempty and has a lower bound zero. Therefore, 
$\inf \mathcal A$ is well-defined. Moreover, by standard approximation argument, 
we can conclude that \eqref{tw solution} also has a monotone solution with $c=\inf \mathcal{A}$ and $a=1$, which implies that $c^{*}(1):=\min \mathcal A\geq0$ is well-defined. 
In addition, in view of Proposition~\ref{prop:existence-tw}, we have $c^{*}(1)\leq c_-$. By Proposition~\ref{prop: minimal tw speed c>0}, we have 
$c^{*}(1)>0$. This completes the proof.
\end{proof}

\begin{remark}
Even though the degenerate case $a=1$ can be seen as the threshold between the monostable and bistable cases, Proposition \ref{cor:c_lv(1)}  indicates that the properties of the degenerate case are more closely related to those of the monostable case.
\end{remark}

\subsection{Asymptotic behavior of traveling waves near $\pm\infty$}

\noindent

In this subsection, we provide the asymptotic behavior of $(U_c,V_c)$ near $\pm\infty$
under assumption {\bf(H1)}, where $(U_c,V_c)$ satisfies \eqref{tw solution} with speed $c$.
Some results are reported in \cite{Kan-On}.
Hereafter, we denote
\beaa
&&\sigma_u^+(c):=\frac{c+\sqrt{c^2+4(a-1)}}{2},\\ &&\sigma_v^{+}(c):=\frac{c+\sqrt{c^2+4rd}}{2d}.
\eeaa

\begin{lemma}[\cite{Kan-On}]\label{lm: behavior around + infty}
Assume that $a>1$ and $b>1$.
Let $(c,U,V)$ be a solution of the system (\ref{tw solution}).
Then there exist positive constants $l_{i=1,\cdots,4}$ such that the following hold:
\beaa
&&\lim_{\xi\rightarrow+\infty}\frac{U(\xi)}{e^{-\sigma_u^+(c)\xi}}=l_1,\\
&&\lim_{\xi\rightarrow+\infty}\frac{1-V(\xi)}{e^{-\sigma_u^+(c)\xi}}=l_2\quad \mbox{if $\sigma_v^{+}(c)>\sigma_u^+(c)$},\\
&&\lim_{\xi\rightarrow+\infty}\frac{1-V(\xi)}{\xi e^{-\sigma_u^+(c)\xi}}=l_3\quad
\mbox{if $\lambda_v^{+}(c)=\sigma_u^+(c)$},\\
&&\lim_{\xi\rightarrow+\infty}\frac{1-V(\xi)}{e^{-\sigma_v^{+}(c)\xi}}=l_4\quad
\mbox{if $\lambda_v^{+}(c)<\sigma_u^+(c)$}.
\eeaa
\end{lemma}

\medskip
To describe the asymptotic behavior of $(U,V)$ near $-\infty$, we define
\beaa
&&
\mu_{u}^{+}(c):=\frac{-c+\sqrt{c^2+4}}{2},\\
&&
\mu_v^{+}(c):=\frac{-c+\sqrt{c^2+4rd(b-1)}}{2d}.
\eeaa

\begin{lemma}[\cite{Kan-On}]\label{lem:AS-infty:b>1}
Assume that $a>1$ and $b>1$.
Let $(c,U,V)$ be a solution of the system (\ref{tw solution}). Then
there exist positive constants $l_{i=5,\cdots,8}$ such that
\beaa
&& \lim_{\xi\rightarrow-\infty}\frac{V(\xi)}{e^{\mu_v^+(c)\xi}}=l_5,\\
&&\lim_{\xi\rightarrow-\infty}\frac{1-U(\xi)}{e^{\mu_v^+(c)\xi}}=l_{6}\quad \mbox{if $\mu_{u}^{+}(c)>\mu_v^+(c)$},\\
&&\lim_{\xi\rightarrow-\infty}\frac{1-U(\xi)}{|\xi|e^{\mu_v^+(c)\xi}}=l_{7}\quad
\mbox{if $\mu_{u}^{+}(c)=\mu_v^+(c)$},\\
&&\lim_{\xi\rightarrow-\infty}\frac{1-U(\xi)}{e^{\mu_u^+(c)\xi}}=l_{8}\quad \mbox{if $\mu_{u}^{+}(c)<\mu_v^+(c)$}.
\eeaa
\end{lemma}

\section{Continuity of the spreading speed}


In this section, 
we shall discuss the continuity of spreading speed of \eqref{system} on the range $a>0$ and $b>1$. 
For any given $b>1$, the system \eqref{system} is of bistable type when $a>1$, which admits a unique traveling wave solution and a unique traveling wave speed $\hat c(a)$ satisfying \eqref{tw solution} (see \cite{Kan-On}). Moreover, it has been proved in \cite{Carrere, Peng Wu Zhou} that the spreading speed coincides with the speed of the bistable front; namely, $c^{**}(a)=\hat c(a)$ for all $a>1$.
On the other hand, for any given $b>1$,
the system \eqref{system} with $a\in(0,1)$ is of monostable type which admits the minimal traveling wave speed $c^*(a)$ satisfying \eqref{tw solution} (see \cite{Kan-on1997}). Furthermore, it has been proved in \cite{Lewis Li Weinberger 2} that $c^*(a)=c^{**}(a)$ for all $a\in(0,1)$. However, for the degenerate case $a=1$, the relation between the spreading speed $c^{**}(1)$ and the minimal traveling speed $c^*(1)$ has not been studied yet. 
Note that, the super and sub-solution method can not be applied for the degenerate case $a=1$ since it is hard to clarify whether the traveling wave decays with an exponential order or a polynomial order.

It is well known from \cite{Lewis Li Weinberger 2} for $0<a<1$ and $b>1$,
and  \cite{Carrere, Peng Wu Zhou} for $a,b>1$ that
$c^{**}(a)$ is continuous on the following region
\beaa
\{a|\, 0<a<1, \}\cup \{a|\, a>1 \}.
\eeaa
Therefore, to complete the proof of Theorem \ref{th: continuity}, it is sufficient to fix $r,d>0$ and $b>1$, and prove $c^{**}(a)$ is continuous at $a=1$. 

\begin{remark}
Notice that, by fixing $b<1$, for all $a\in\mathbb{R}^+$, there exists a minimal traveling wave speed that always corresponds to the spreading speed, but with different boundary conditions as those in \eqref{tw solution}. For more details, see \cite{Wu Xiao Zhou 2}. The continuity of $c^{**}(a)$ with $b<1$ follows immediately from the approximation argument.
Notice that the propagation phenomenon of the critical case $a=b=1$ has been well studied in \cite{Alfaro Xiao}. Due to the degeneracy on both $a=1$ and $b=1$, the propagation property of the solution does not depend on $a$ and $b$, but on the value of $dr$. As a result, it is not sensible to consider the continuity on $a=b=1$, and we exclude it from our discussion. 
\end{remark}

Now, we fix $b>1$, $r,d>0$, and establish the continuity of the spreading speed of \eqref{system} on $a\in(0,\infty)$. 
 It suffices to discuss the continuity of 
the spreading speed $c^{**}(a)$
at $a=1$ on both direction $a\to 1^+$ and $a\to 1^-$. 
Recall from \cite{Kan-On} that, for $a>1$, 
the system \eqref{system} admits a unique traveling wave solution with unique speed $\hat{c}(a)\in[-2\sqrt{rd},2]$. Moreover, 
$\hat{c}(a)$ is strictly decreasing for $a>1$.
Therefore, we can 
define 
\beaa
c_+:=\lim_{a\to 1^+}\hat c(a).
\eeaa

To prove the continuity of 
$c^{**}(a)$,
we first need to figure out the relation 
among $c^*(1)$, $c_+$ and $c_-$,
where $c_-$ is defined in \eqref{c-}.

\begin{lemma}\label{prop: c**<hat c}
Let $r,d>0$ and $b>1$ be given. 
It holds that 
\beaa
c_+\le c^*(1)\le c_-.
\eeaa
\end{lemma}
\begin{proof}
The relation between $c^*(1)$ and $c_-$ follows from Proposition~\ref{cor:c_lv(1)}.
We now show that $c_+\le c^*(1)$ using
some comparison argument. Assume by contradiction that $c^*(1)< c_+$. Then,
from the continuity and the monotonicity of $\hat{c}(a)$ for $a>1$,
there exists $\delta_0>0$ small enough such that the traveling wave speed of \eqref{system} with $a=1+\delta_0$ satisfying 
\bea\label{order-c-hat}
c^*(1)<\hat c(1+\delta_0)<c_{+}.
\eea
Denote $(c^*(1), U_1,V_1)$ as the traveling wave solution of \eqref{system} with $a=1$. It is easy to check that 
$$\Big(U_1(x-c^*(1)t-x_0),V_1(x-c^*(1)t-x_0)\Big)$$ is a super-solution of
\begin{equation}\label{system a=1+delta}
    \left\{\begin{aligned}\relax
    & u_t=u_{xx}+u(1-u-(1+\delta_0)v) \\ 
        & v_t=dv_{xx}+rv(1-bu-v). 
    \end{aligned}\right.
\end{equation}
It has been proved in Theorem 3 of \cite{Peng Wu Zhou} that, the spreading speed of \eqref{system a=1+delta} with the initial datum \eqref{initial datum} is  equal to 
$\hat c(1+\delta_0)$.

On the other hand, by setting $x_0>0$ large enough, we have 
$$ U_1(x-x_0)\ge u_0(x)\ \ \text{and}\ \  V_1(x-x_0)\le v_0(x)\ \text{for}\ x\in \mathbb{R}.$$
with $(u_0(x),v_0(x))$ satisfying \eqref{initial datum}. Then, by the comparison principle, we get 
$$U_1(x-c^*(1)t-x_0)\ge u(t,x)\ \ \text{and}\ \ V_1(x-c^*(1)t-x_0)\le v(t,x)$$
for $t>0$ and $x\in\mathbb{R}$,
which implies the spreading speed $\hat c(1+\delta_0)$ of \eqref{system a=1+delta} is smaller than or equal to $c^*(1)$. However, this contradicts to \eqref{order-c-hat} and completes the proof.
\end{proof}

In order to complete the proof of the continuity of the minimal speed $c^*(a)$ at $a=1$, we just
need to verify $c_-=c_+$.

\begin{proposition}\label{prop: c**=c*}
Let $r,d>0$ and $b>1$ be given. 
It holds that 
\beaa
c_+ = c^*(1)= c_-.
\eeaa
\end{proposition}
\begin{proof}
From Lemma~\ref{prop: c**<hat c}, we have $c_-\ge  c_+$. To prove $c_-=c_+$, we 
assume by contradiction that 
\bea\label{c->c+}
c_->c_+.
\eea
To emphasize the dependence on both $a,b>1$, we denote $\hat c:=\hat c(a,b)$. Since $\hat c(a,b)$ is strictly decreasing on $a\in(1,\infty)$ and strictly increasing on $b\in(1,\infty)$, we can choose very small $\delta_*>0$ such that 
\bea\label{c1 and c+}
c_1:=\hat c(1+\delta_*,b)< c_+.
\eea
Then we choose
\begin{equation}\label{c2}
c_2=c_++\frac{(c_--c_+)}{2}\in (c_+,c_-)
\end{equation}
such that $c_1<c_2< c_-$.

Next, we show that there exists a super-solution $(\overline U,\underline V)(\xi)$ satisfying 
\begin{equation}\label{def: super solution 1}
\left\{
\begin{aligned}
&N_1[\overline U,\underline V]=\overline U''+c_2\overline U'+\overline U(1-\overline U-(1-\delta_0)\underline V)\le 0,\\
&N_2[\overline U,\underline V]=d\underline V''+c_2\underline V'+r\underline V(1-\underline V-b\overline U)\ge 0,
\end{aligned}
\right.
\end{equation}
for some small $\delta_0>0$.


\begin{figure}
\begin{center}
\begin{tikzpicture}[scale = 1.1]
\draw[thick](-6,0) -- (6,0) node[right] {$\xi$};
\draw [ultra thick] (-6,-0.7)to [out=0,in=140] (-3,-1.5) -- (-1.7,-1.5)  to [out=30,in=270]   (-1,0)  to [out=90,in=185]  (1,3) to [out=45,in=190] (1.5,3.3) to [out=0,in=160] (6,0.5);
\draw [semithick] (-6, -0.1) to [ out=0, in=150] (-4,-0.5) to [out=30, in=230] (-3.5,0)  to [out= 40, in=180] (-3,0.2) to [out=15, in=220] (1,1.6) to [out=70,in=180] (2,2)to [out=0,in=170] (6,0.2);
\node[below] at (1,0) {$\xi_{1}+\delta_1$};
\draw[dashed] [thick] (-3.5,0)-- (-3.5,-0.5);
\node[below] at (-3.5,-0.4) {$\xi_2$};
\draw[dashed] [thick] (-3,0.2)-- (-3,-0.5);
\node[below] at (-2.7,-0.4) {$\xi_2+\delta_5$};
\draw[dashed] [thick] (-4,0)-- (-4,1);
\node[above] at (-4,0.8) {$\xi_2-\delta_7$};
\draw[dashed] [thick] (1,0)-- (1,1.6);
\draw[dashed] [thick] (1,2)-- (1,3);
\draw[dashed] [thick] (-1,0)-- (-1,-1.1);
\draw[dashed] [thick] (-1.7,-0.9)-- (-1.7,1.7);
\node[above] at (-1.7,1.55) {$\xi_1-\delta_4$};
\node[below] at (-0.7,-1) {$\xi_1-\delta_3$};
\draw [thin] (-3.85,-0.41) arc [radius=0.2, start angle=40, end angle= 140];
\node[above] at (-4,-0.43) {$\alpha_6$};
\draw [thin] (-2.8,-1.5) arc [radius=0.2, start angle=0, end angle= 135];
\node[above] at (-2.8,-1.4) {$\alpha_5$};
\draw [thin] (-2.8,0.255) arc [radius=0.2, start angle=30, end angle= 175];
\node[above] at (-3,0.3) {$\alpha_4$};
\draw [thin] (-1.5,-1.35) arc [radius=0.2, start angle=20, end angle= 200];
\node[above] at (-1.7,-1.3) {$\alpha_3$};
\draw [thin] (1.1,1.8) arc [radius=0.2, start angle=70, end angle= 220];
\node[above] at (0.7,1.7) {$\alpha_1$};
\draw [thin] (1.1,3.13) arc [radius=0.2, start angle=50, end angle= 180];
\node[above] at (1,3.13) {$\alpha_2$};
\node[above] at (2,2.3) {\Large{$R_u$}};
\node[above] at (2,1) {\Large{$R_v$}};
\end{tikzpicture}
\caption{auxiliary functions $(R_u,R_v)$.}\label{Figure:b>1}
\end{center}
\end{figure}
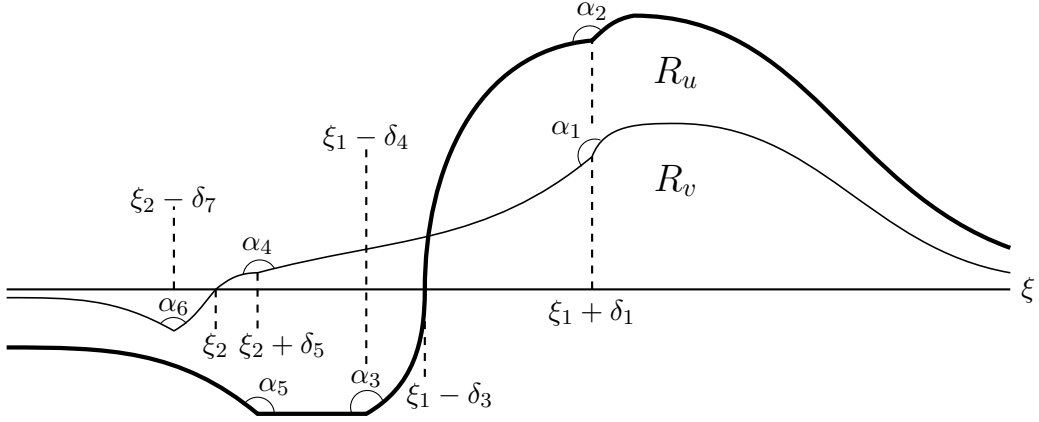

\bigskip
\noindent
{\bf{Construction of the super-solution}}
Define $(U_0,V_0)$ as the unique traveling wave solution of \eqref{tw solution} with $a=1+\delta_*$ and $c=c_1$. 
Let $(\overline U,\underline V)(\xi)=(U_0-R_u,V_0+R_v)(\xi)$ where  the auxiliary functions $(R_u,R_v)(\xi)$ (see Figure \ref{Figure:b>1}) are defined as  
\begin{equation*}
(R_u,R_v)(\xi):=\begin{cases}
(\varepsilon_1(\xi-\xi_1) e^{-\lambda_1\xi},\eta_1 (\xi-\xi_1)e^{-\lambda_1\xi}),&\ \ \mbox{for}\ \ \xi_1+\delta_1\le\xi,\\
(\varepsilon_2\sin(\delta_2(\xi-\xi_1+\delta_3)),\eta_2e^{\lambda_{1}\xi}),&\ \ \mbox{for}\ \ \xi_1-\delta_4\le\xi\le\xi_{1}+\delta_1,\\
(-\varepsilon_3,\eta_2e^{\lambda_{2}\xi}),&\ \ \mbox{for}\ \ \xi_2+\delta_5\le\xi\le\xi_1-\delta_4,\\
(-\varepsilon_4e^{\lambda_3\xi},\eta_3\sin(\delta_6(\xi-\xi_2))),&\ \ \mbox{for}\ \ \xi_2-\delta_7\le\xi\le\xi_2+\delta_5,\\
(-\varepsilon_4e^{\lambda_3\xi},-\eta_4e^{\lambda_3\xi}),&\ \ \mbox{for}\ \ \xi\le\xi_2-\delta_7,
\end{cases}
\end{equation*}
where
$\lambda_1>\sigma_u^+$ ($\sigma_u^+=\sigma_u^+(c_1,1+\delta_*$) is defined in Lemma \ref{lm: behavior around + infty}), 
sufficiently large $\lambda_{2}>0$, and sufficient small
\begin{equation}\label{condition on lambda 2}
0<\lambda_3<\min\{\mu_u^+,\mu_v^+\}\quad (\mu_u^+=\mu_u^+(c_1),\mu_v^+=\mu_v^+(c_1)\ \ \text{are defined in Lemma \ref{lem:AS-infty:b>1}}). 
\end{equation}

Since
$(U_0,V_0)(-\infty)=(1,0)$ and
$(U_0,V_0)(+\infty)=(0,1)$, for any given small $\rho>0$, we can take $M_0>0$ sufficiently large such that
\bea\label{rho}
\begin{cases}
0< U_0(\xi)<\rho,\quad 1-\rho<V_0(\xi)<1&\quad\mbox{for all}\quad\xi\geq M_0,\\
0< V_0(\xi)<\rho,\quad 1-\rho<U_0(\xi)<1&\quad\mbox{for all}\quad\xi\leq -M_0.
\end{cases}
\eea
We set $\xi_1>M_0$ and $\xi_2<-M_0$ with $M_0$ defined as \eqref{rho}, and set $\varepsilon_{i=1,\cdots,4}>0$ and $\eta_{j=1,\cdots,4}>0$ to make $(R_u,R_v)(\xi)$ continuous on $\mathbb{R}$. $\delta_{k=1,\cdots,7}>0$ will be determined later which make $(\overline U, \underline V)$ satisfy \eqref{def: super solution 1}.

\medskip

Next, we divide the proof into several steps.
\medskip

\noindent{\bf{Step 1}:} We consider $\xi\in[\xi_{1}+\delta_1,\infty)$ with $\xi_1>M_0$  and  some small $\delta_1$ satisfying
\bea\label{cond delta 1 pushed}
0<\delta_1<\frac{1}{\lambda_1+\lambda_2}.
\eea
In Step 1, we aim to verify that $(\overline U,\underline V)(\xi)=(U_0-R_u,V_0+R_v)(\xi)$ with
\beaa
(R_u,R_v)(\xi)=(\varepsilon_1(\xi-\xi_1)e^{-\lambda_1\xi},\eta_1(\xi-\xi_{1}) e^{-\lambda_1\xi}),
\eeaa
satisfies \eqref{def: super solution 1} by setting $\eta_1\ll \varepsilon_1\ll 1$.
With \eqref{cond delta 1 pushed}, it is easy to check $R'_u((\xi_1+\delta_1)^+)>0$  and $R'_v((\xi_1+\delta_1)^+)>0$.


By some  straightforward computations,  we have
\begin{equation*}
\begin{aligned}
N_1[\overline U,\underline V]=&-R_u(\lambda_1^2-\lambda_1c_2+1-2U_0+R_u-(1-\delta_0)(V+R_v))\\
&+(c_2-c_1)U'_0+(\delta_0+\delta_*)U_0V_0+(2\lambda_1-c_2)\varepsilon_1e^{-\lambda_1\xi}-(1-\delta_0)U_0R_v.
\end{aligned}
\end{equation*}
From Lemma \ref{lm: behavior around + infty}, we have $-U_0'/U_0\sim \sigma_u^+$ 
as $\xi\to\infty$. 
Note that 
$\sigma_u^+\to c_+>c_1$ as $\delta_*\to 0$ due to \eqref{c1 and c+} and Lemma \ref{lm: behavior around + infty}. Then, by setting $\delta_0,\delta_*>0$ small enough such that
$$\delta_0+\delta_*<\Big(\frac{c_{-} -c_+}{4}\Big)c_1,$$
from \eqref{c2} there exists $C_1>0$ such that
\beaa
(c_2-c_1)U'_0(\xi)+(\delta_0+\delta_*)U_0(\xi)V_0(\xi)<-C_1U'_0(\xi), \quad \xi \ge \xi_1+\delta_1.
\eeaa
Here we also used the fact that $c_2-c_1>(c_--c_+)/2$. 
Moreover, since $\lambda_1>\sigma_u^+$, we have 
\beaa
e^{-\lambda_1\xi}=o(U'_0(\xi)),\quad R_v(\xi)=o(U'_0(\xi)),\quad {\rm and}\quad R_u(\xi)=o(U'_0(\xi))\quad \mbox{as}\quad \xi\to\infty.
\eeaa
Then, up to enlarging $M_0$ if necessary, we obtain that $N_1[\overline U,\underline V]\leq0$
for all $\xi\in[\xi_1+\delta_1,\infty)$.

Next, we deal with the inequality of $N_2[\overline U, \underline V]$. By some  straightforward computations,  we have
\begin{equation*}
\begin{aligned}
N_2[\overline U, \underline V] = & (c_2-c_1)V'_0+(c_2-2d\lambda_1)\eta_1e^{-\lambda_1\xi}+rbV_0R_u\\
&+ rR_v\Big[\frac{d\lambda_1^2-c_2\lambda_1}{r}+1-2V_0-R_v-b(U_0-R_u)\Big].
\end{aligned}
\end{equation*}
From $\lambda_1>\sigma_u^+$ and Lemma \ref{lm: behavior around + infty}, we have
 \beaa
 e^{-\lambda_1\xi}=o(V'_0(\xi))\quad {\rm and}\quad  R_v(\xi)=o(V'_0(\xi))\quad \mbox{as}\quad \xi\to\infty.
 \eeaa
 Then, up to enlarging $M_0$ if necessary, we obtain that $N_2[\overline U,\underline V]\geq0$
for all $\xi\in[\xi_1+\delta_1,\infty)$.

\medskip

\noindent{\bf{Step 2:}} We consider $\xi\in[\xi_1-\delta_4,\xi_1+\delta_1]$ with $\delta_1$ satisfying \eqref{cond delta 1 pushed} and $\delta_4>0$ will be determined later. In this case, we have
$$(R_u,R_v)(\xi):=(\varepsilon_2\sin(\delta_2(\xi-\xi_1+\delta_3)),\eta_2e^{\lambda_{2}\xi})$$
with $0<\delta_2<\lambda_2(\delta_1+\delta_3)$ sufficiently small and $\delta_1+\delta_3<\pi/2\delta_2$.

We first verify the following claim: 
\begin{claim}\label{cl: alpha 1 aplha 2 pushed}
There exist 
$\varepsilon_2>0$ and $\eta_2>0$ sufficiently small
such that
\beaa
&&R_u((\xi_{1}+\delta_1)^+)=R_u((\xi_{1}+\delta_1)^-),\quad R_v((\xi_{1}+\delta_1)^+)=R_v((\xi_{1}+\delta_1)^-),\\ 
&&R'_u((\xi_{1}+\delta_1)^+)>R'_u((\xi_{1}+\delta_1)^-), \quad\text{and}\quad R'_v((\xi_{1}+\delta_1)^+)>R'_v((\xi_{1}+\delta_1)^-),
\eeaa
provided that  $\delta_1$ satisfy \eqref{cond delta 1 pushed},  and $\delta_2$ is sufficiently small.
Consequently, we have $\angle\alpha_{1,2}<180^{\circ}$.
\end{claim}
\begin{proof}
By some  straightforward computations, we have
\beaa
&&R_u((\xi_{1}+\delta_1)^+)=\varepsilon_1\delta_1e^{-\lambda_1(\xi_{1}+\delta_1)},\quad R_u((\xi_{1}+\delta_1)^-)=\varepsilon_2\sin(\delta_2(\delta_1+\delta_3)),\\
&&R'_u((\xi_{1}+\delta_1)^+)=\varepsilon_1e^{-\lambda_1(\xi_{1}+\delta_1)}-\lambda_1R_u(\xi_{1}+\delta_1),\\
&&R'_u((\xi_{1}+\delta_1)^-)=\varepsilon_2\delta_2\cos(\delta_2(\delta_1+\delta_3)).
\eeaa
We first choose 
$\varepsilon_2=\varepsilon_1\delta_1e^{-\lambda_1(\xi_{1}+\delta_1)}/\sin(\delta_2(\delta_1+\delta_3))$ such that
\beaa
R_u((\xi_{1}+\delta_1)^+)=R_u((\xi_{1}+\delta_1)^-).
\eeaa
Then, by applying  the fact $\frac{x\cos x}{\sin x}\to 1$ as $x\to 0$, we have 
$$R'_u((\xi_{1}+\delta_1)^+)-R'_u((\xi_{1}+\delta_1)^-)>0$$
 is equivalent to
\beaa
\delta_1(\lambda_1+\frac{\delta_2}{\delta_1+\delta_3})<1,
\eeaa
which holds since \eqref{cond delta 1 pushed} and $\delta_2<\lambda_2(\delta_1+\delta_3)$ very small.

On the other hand, by some  straightforward computations, we have
\beaa
&&R_v((\xi_{1}+\delta_1)^-)=\eta_2e^{\lambda_2(\xi_{1}+\delta_1)},\ R_v((\xi_{1}+\delta_1)^+)=\eta_1\delta_1e^{-\lambda_1(\xi_{1}+\delta_1)},\\
&&R'_v((\xi_{1}+\delta_1)^-)=\lambda_2\eta_2e^{\lambda_2(\xi_{1}+\delta_1)},\ R'_v((\xi_{1}+\delta_1)^+)=\eta_1(1-\delta_1\lambda_1)e^{-\lambda_1(\xi_{1}+\delta_1)},
\eeaa
where $\lambda_2$ is large.
We take
\beaa
\eta_2=\eta_1\delta_1e^{-(\lambda_1+\lambda_2)(\xi_{1}+\delta_1)}>0,
\eeaa
which implies $R_v((\xi_{1}+\delta_1)^-)=R_v((\xi_{1}+\delta_1)^+)$. Then
\beaa
R_v'((\xi_{1}+\delta_1)^+)-R_v'((\xi_{1}+\delta_1)^-)=\eta_1e^{-\lambda_1(\xi_1+\delta_1)}(1-\delta_1\lambda_1-\delta_1\lambda_2)>0,
\eeaa
is equivalent to \eqref{cond delta 1 pushed}.
The proof of Claim \ref{cl: alpha 1 aplha 2 pushed} is complete.
\end{proof}

To finish Step 2, it suffices to take a small $\delta_2>0$ and suitable $0<\delta_3<\delta_4$ such that
\bea\label{goal: step 3}
N_1[\overline U,\underline V]\leq 0\quad\text{and}\quad N_2[\overline U,\underline V]\geq 0\quad \mbox{for}\quad \xi\in[\xi_1-\delta_4,\xi_1+\delta_1].
\eea
By some  straightforward computations, for $\xi\in[\xi_2-\delta_4,\xi_1+\delta_1]$ we have
\beaa
N_1[\overline U,\underline V]&=&-R''_u-c_2R'_u-(c_1-c_2)U_0'+(\delta_0+\delta_*)U_0V_0\\
&&-R_u(1-2U_0+R_u-(1-\delta_0)(V_0+R_v))-(1-\delta_0)U_0R_v,\\
N_2[\overline U,\underline V]&=&dR''_v+c_2R'_v+rR_v[1-2V_0-R_v-b(U_0-R_u)]
\\&&+rbV_0R_u+(c_2-c_1)V_0'.
\eeaa


Notice that, $c_2-c_1>(c_--c_+)/2$ which is independent on $\delta_0,\delta_*$ and $|R_u|,|R_v|$. Therefore
$(c_1-c_2)U_0'$ and $(c_2-c_1)V_0'$ have a lower bound independent on $\delta_0,\delta_*$ and $|R_u|,|R_v|$ for $\xi\in[\xi_2-\delta_4,\xi_1+\delta_1]$.
Since $R_u$ and $R_v$ are set to be continuous for all $\xi\in\mathbb{R}$, 
$$|R_u|,|R'_u|, |R''_u|\to 0\quad \text{and}\quad |R_v|,|R'_v|, |R''_v|\to 0\quad  \text{as} \quad\varepsilon_1,\eta_1\to 0.$$
As a result, we obtain $N_1[\overline U,\underline V]\le 0$ and $N_2[\overline U,\underline V]\ge 0$ up to decreasing $\delta_0$, $\delta_*$, $\varepsilon_1$, and $\eta_1$ if necessary.
This completes the proof of \eqref{goal: step 3}, and Step 2 is finished.

\medskip
    
\noindent{\bf{Step 3:}} We consider $\xi\in[\xi_2+\delta_5,\xi_1-\delta_4]$ for some $\delta_5>0$ such that $\xi_2+\delta_5<-M_0$. We set $\delta_5$ to satisfy
\bea\label{cond delta 5}
\delta_5>\frac{1}{\lambda_2}.
\eea
In this case, we have
\beaa
(R_u,R_v)(\xi)=(-\varepsilon_3,\eta_2e^{\lambda_2\xi}).
\eeaa

First, we choose $\varepsilon_3=R_u(\xi_1-\delta_4)$ such that $R_u(\xi)$ is continuous at $\xi=\xi_1-\delta_4$.
Clearly, by setting $|\delta_3-\delta_4|$ very small as in Step 2, we have 
$$R'_u((\xi_1-\delta_{4})^+)>0=R'_u((\xi_1-\delta_{4})^-),\  {\it i.e.},\  \angle\alpha_3<180^{\circ}.$$

Similar to Step 2,  $N_1[\overline U,\underline V]\le 0$ and $N_2[\overline U,\underline V]\ge 0$ hold up to decreasing $\delta_0$, $\delta_*$, $\varepsilon_1$, and $\eta_1$ if necessary.

\medskip

\noindent{\bf{Step 4:}}
We consider $\xi\in[\xi_2-\delta_7,\xi_2+\delta_5]$.
In this case, we have
\beaa
(R_u,R_v)(\xi)=\Big(-\varepsilon_4e^{\lambda_3\xi},\eta_3\sin(\delta_6(\xi-\xi_2))\Big),
\eeaa
where  $\varepsilon_4>0$, $\eta_3>0$,
$\delta_6>0$, and $\delta_7>0$ are determined below.



We first choose
\beaa
\varepsilon_4=\frac{\varepsilon_3}{e^{\lambda_3(\xi_2+\delta_5)}}
\eeaa
such that $R_u(\xi)$ is continuous at $\xi=\xi_2+\delta_5$,
where $\varepsilon_3$ and $\delta_5$ are fixed in Step 3.
Then, 
we have
\beaa
R'_u((\xi_2+\delta_5)^+)=0>R'_u((\xi_2+\delta_5)^-),
\eeaa
and thus $\angle\alpha_5<180^{\circ}$.
Next, we verify the continuity of $R_v$ and the right angle of $\alpha_4$.
\begin{claim}\label{cl 5}
For any $\delta_5$ satisfying \eqref{cond delta 5}, there exist
$\eta_3>0$ and small $\delta_6>0$
such that $R_v(\xi)$ is continuous at $\xi=\xi_2+\delta_5$ and $\angle\alpha_4<180^{\circ}$.
\end{claim}
\begin{proof}
First, we  take
\bea\label{eta4}
\eta_3= \frac{\eta_2e^{\lambda_2(\xi_2+\delta_5)}}{\sin(\delta_5\delta_6)}>0
\eea
such that $R_v((\xi_2+\delta_5)^+)=R_v((\xi_2+\delta_5)^-)$.

By some straightforward computations, we have
$$R'_v((\xi_2+\delta_5)^+)=\lambda_2\eta_2e^{\lambda_2(\xi_2+\delta_5)}.$$ 
Then from \eqref{eta4},
\beaa
R'_v((\xi_2+\delta_5)^-)=\eta_3\delta_6\cos(\delta_5\delta_6)=\eta_2e^{\lambda_2(\xi_2+\delta_5)}\frac{\delta_6\cos(\delta_5\delta_6)}{\sin(\delta_5\delta_6)},
\eeaa
which yields that
$$R'_v((\xi_2+\delta_5)^-)\rightarrow \eta_2e^{\lambda_2(\xi_2+\delta_5)}/\delta_5\quad\text{as}\quad \delta_6\to0.$$
Thus, $R'_v((\xi_2+\delta_5)^+)>R'_v((\xi_2+\delta_5)^-)$ is equivalent to $\delta_5>\frac{1}{\lambda_2}$ by setting $\delta_6$ sufficiently small.
This completes the proof of Claim~\ref{cl 5}.
\end{proof}

Similar to Step 2,  $N_1[\overline U,\underline V]\le 0$ and $N_2[\overline U,\underline V]\ge 0$ hold up to decreasing $\delta_0$, $\delta_*$, $\varepsilon_1$, and $\eta_1$ if necessary.

\medskip

\noindent{\bf{Step 5:}} We consider $\xi\in(-\infty,\xi_2-\delta_7]$.
In this case, we have
\beaa
(R_u,R_v)(\xi)=\Big(-\varepsilon_4e^{\lambda_3\xi},-\eta_4e^{\lambda_3\xi}\Big).
\eeaa
Let us take $$\eta_4=\frac{\eta_3\sin(\delta_6\delta_7)}{e^{\lambda_3(\xi_2-\delta_7)}}$$
 such that
$R_v(\xi)$ is continuous at $\xi=\xi_2-\delta_7$.
Also, since $0<\delta_7\le \delta_5$ and $R_v(\xi)$ is decreasing on $\xi$ for $\xi\le\xi_2-\delta_7$, we have
$$R'_v((\xi_2-\delta_7)^+)>0>R'_v((\xi_2-\delta_7)^-),$$ 
and hence $\angle\alpha_6<180^{\circ}$.

Finally, we verify the differentiable inequalities.
Due \eqref{condition on lambda 2} and Lemma \ref{lem:AS-infty:b>1}, there exists $M_1>M_0$ sufficiently large such that $\overline U=1$ and $\underline V=0$ for all $\xi\in(-\infty,-M_1]$.
As a result, it suffices to verify the inequalities for $\xi\in[-M_1,\xi_2-\delta_7]$.
By some  straightforward computations and \eqref{rho},  we have
\beaa
N_1[\overline U,\underline V]&=&-R''_u-c_2R'_u-(c_1-c_2)U_0'+(\delta_0+\delta_*)U_0V_0\\
&&-R_u(1-2U_0+R_u-(1-\delta_0)(V_0+R_v))-(1-\delta_0)U_0R_v\\
&\le & -(\lambda_3^2+c_2\lambda_3-1+4\rho+\eta_4/\varepsilon_4)R_u-(c_1-c_2)U_0'+(\delta_0+\delta_*)V_0.
\eeaa

Notice that, $\eta_4/\varepsilon_4\to 0$ as $\delta_7\to 0$. Then, up to decreasing $\lambda_3$ and increasing $M_0$ if necessary, we have 
$$\lambda_3+c_2\lambda_3+4\rho+\eta_4/\varepsilon_4<1.$$ 
Moreover, by Lemma \ref{lem:AS-infty:b>1}, up to decreasing $\delta_0+\delta_*$ if necessary, we obtain 
$$(c_2-c_1)U_0'+(\delta_0+\delta_*)V_0\le 0.$$
Consequently, we can conclude that $N_1[\overline U,\underline V]\le 0$ for $\xi\in[-M_1,\xi_2-\delta_7]$, Provided that $\delta_0$, $\delta_7$, and $\delta_*$ are sufficient small.

On the other hand, by some  straightforward computations,  we have
\beaa
N_2[\overline U,\underline V]&=&R''_v+c_2R'_v+rR_v[1-2V_0-R_v-b(U_0-R_u)]
\\&&+rbV_0R_u+(c_2-c_1)V_0'\\
&\ge & [d\lambda_3^2+c_2\lambda_3-(b-1)r+r(b+1)\rho]R_v+rbV_0R_u+(c_2-c_1)V'_0.
\eeaa
Then, up to decreasing $\lambda_3$ and increasing $M_0$ if necessary, we have 
$$d\lambda_3^2+c_2\lambda_3-(b-1)r+r(b+1)\rho<0.$$ 
Moreover, by Lemma \ref{lem:AS-infty:b>1}, up to increasing $M_0$ if necessary, we obtain 
 $$rbV_0R_u+(c_2-c_1)V'_0\ge 0.$$
Consequently, we can conclude that $N_2[\overline U,\underline V]\ge 0$ for $\xi\in[-M_1,\xi_2-\delta_7]$. The verification for Step 5 is complete.

\medskip

Now, we are ready to finish the proof of Proposition \ref{prop: c**=c*}. For any given $b>1$, the spreading speed of \eqref{system} with $a=1-\delta_0$ and initial datum satisfying \eqref{initial datum}  is equal to the minimal traveling wave speed $c^*(1-\delta_0)>c_-$. On the other hand, we define 
$$(\overline u,\underline v)(t,x):=(\overline U,\underline V)(x-c_2t-x_0)\ \text{with}\ x_0>0.$$ 
By setting $x_0>0$ large enough, we have $\overline u(0,x)\ge u_0(x)$ and $\underline v(0,x)\le v_0(x)$ where $(u_0,v_0)$ satisfies \eqref{initial datum} with $v_0\equiv 1$. Then, by applying the comparison principle, the spreading speed is smaller than or equal to $c_2$, {\it i.e.}, $c^*(1-\delta_0)\leq c_2$. This contradicts to $c_2<c_-<c^*(1-\delta_0)$. Therefore, \eqref{c->c+} is impossible and so we conclude that $c_+=c_-$. 
\end{proof}

\begin{remark}
Due to Proposition \ref{prop: minimal tw speed c>0} and the definition of $c^*(1)$, one has $c^*(1)>0$. Then, Theorem \ref{th:speed selection} follows immediately by the continuity of $c^*(a)$.
\end{remark}

\section{Proof of Theorem \ref{th:conditions} and \ref{th: continuity}}

Now, we are ready to prove Theorem \ref{th: continuity}.

\begin{proof}[Proof of Theorem \ref{th: continuity}]
 In view of Proposition~\ref{prop: c**=c*}, we have $c_-=c^*(1)=c_+$. Then, by the comparison argument, we can show that the spreading speed  $c^{**}(1)$ exists and $c^{**}(1)=c^{*}(1)$. To see this, for any given small $\epsilon>0$, we consider 
\bea\label{system a=1-e}
\left\{
\begin{aligned}
&\tilde{u}_t=\tilde{u}_{xx}+\tilde{u}(1-\tilde{u}-(1-\epsilon)\tilde{v}), & t>0,\ x\in\mathbb{R},\\
&\tilde{v}_t=d\tilde{v}_{xx}+r\tilde{v}(1-\tilde{v}-b\tilde{u}), & t>0,\ x\in \mathbb{R},
\end{aligned}
\right.
\eea
with initial datum satisfies \eqref{initial datum}. Note that, by choosing $\epsilon>0$ sufficiently small, Theorem \ref{th:speed selection} implies $c^*(1-\epsilon)$ is nonlinearly selected.
Then by using Theorem 1.1 in \cite{Wu Xiao Zhou}, we have
\bea\label{long-time-mono}
\sup_{x\geq0}|(\tilde{u},\tilde{v})(t,x)-(\tilde{U},\tilde{V})(x- c^{*}(1-\epsilon)t-x_0)|\to0
\eea
as $t\to\infty$, where $x_0$ is some constant and
$(c^{*}(1-\epsilon), \tilde{U}, \tilde{V})$ is the minimal traveling wave of 
\eqref{system a=1-e} with
\beaa
(\tilde{U}, \tilde{V})(-\infty)=(1,0),\quad (\tilde{U}, \tilde{V})(+\infty)=(0,1).
\eeaa
On the other hand,  we consider 
\bea\label{system a=1+e}
\left\{
\begin{aligned}
&\check{u}_t=\check{u}_{xx}+\check{u}(1-\check{u}-(1+\epsilon)\check{v}), & t>0,\ x\in\mathbb{R},\\
&\check{v}_t=d\check{v}_{xx}+r\check{v}(1-\check{v}-b\check{u}), & t>0,\ x\in \mathbb{R},
\end{aligned}
\right.
\eea
with initial datum satisfies \eqref{initial datum}, which forms a bistable system. Up to decreasing $\epsilon$ if necessary, $\hat c(1+\epsilon)>0$ from Proposition \ref{prop: c**=c*}.
Then by using Theorem 1 in \cite{Peng Wu Zhou}, we have
\bea\label{long-time-bis}
\sup_{x\geq0}|(\check{u},\check{v})(t,x)-(\check{U},\check{V})(x- c^{*}_{LV}(1-\epsilon)t-x_1)|\to0
\eea
$t\to\infty$, 
where $x_1$ is some constant and $( \hat c(1+\epsilon), \check{U}, \check{V})$ is the unique bistable type traveling wave of \eqref{system a=1+e} with
\beaa
(\check{U}, \check{V})(-\infty)=(1,0),\quad (\check{U}, \check{V})(+\infty)=(0,1).
\eeaa

By comparison with the solution $(u,v)$ to 
\beaa
\left\{
\begin{aligned}
&u_t=u_{xx}+u(1-u-(1-\epsilon)v), & t>0,\ x\in\mathbb{R},\\
&v_t=dv_{xx}+rv(1-v-bu), & t>0,\ x\in \mathbb{R},
\end{aligned}
\right.
\eeaa
with initial datum \eqref{initial datum}, we have 
\bea\label{CP-u}
\check{u}(t,x)\leq u(t,x) \leq \tilde{u}(t,x),\quad \check{v}(t,x)\geq v(t,x) \geq \tilde{v}(t,x)
\eea
for $t>0$ and $x\in\mathbb{R}$. Since $\epsilon>0$ is arbitrary, combining \eqref{long-time-mono}, \eqref{long-time-bis},  \eqref{CP-u} and the fact that
$c_-=c^*(1)=c_+$,
we conclude that the spreading speed exists and is equal to $c^*(1)$.
Consequently, $c^{**}(1)=c^*(1)$.
This completes the proof of Theorem \ref{th: continuity}.
\end{proof}

Last, we complete the proof of Theorem \ref{th:conditions}.

\begin{proof}[Proof of Theorem \ref{th:conditions}]
The existence of $a_2$ follows immediately from Theorem \ref{th: continuity}.
Now, we deal with the existence of $a_1$.
To get the contradiction, we assume there exist $\{a_n\}_{n\in\mathbb{N}}$ with $a_n\to\infty$ such that $\hat c(a_n)\le 0$ for all $n\in\mathbb{N}$.
Due to the monotonicity of $\hat c(a_n)$, we assume without generality that $\hat c(a_n)=0$ for all $n\in\mathbb{N}$.
We write $(U_n,V_n)$ as the solution of \eqref{tw solution} with $c=\hat c(a_n)= 0$.
By a translation, we may assume that $V_n(0)=1/2$ for all $n$.
Since $0\leq U_n,V_n\leq 1$ in $\mathbb{R}$, by standard elliptic estimates, we have
$\|V_n\|_{C^{2+\alpha}(\mathbb{R})}\leq C$ for some $C>0$ independent of $n$.

We now fix $R>0$. Then there exists $\varepsilon>0$ such that
\bea\label{U-lower-bdd}
V_n(\xi)\geq \varepsilon\quad \mbox{for all $\xi\in[-R,R]$ and $n\in\mathbb{N}$}.
\eea
Next, we define an auxiliary function
\beaa
\bar{U}_n(\xi)=\frac{e^{-\lambda_n(\xi+2R)}+e^{\lambda_n(\xi-2R)}}{1+e^{-4\lambda_nR}},\quad \xi\in[-2R,2R],
\eeaa
where
\beaa
\lambda_n:=\sqrt{a_n\varepsilon}\to\infty\quad\mbox{as}\quad n\to\infty.
\eeaa
Clearly, $\bar{U}_n(\pm 2R)=1$, $0\leq \bar{U}_n(\xi)\leq 1$ for all $\xi\in[-2R,2R]$ and $n\in\mathbb{N}$, and $\bar{U}_n\to0$ uniformly in $[-R,R]$ as $n\to\infty$.
Furthermore, by direct computation, for all large $n$ we have
\beaa
\bar{U}''_n+\bar{U}_n(1-\bar{U}_n)-a_n\varepsilon \bar{U}_n\leq 0,\quad \xi\in[-2R,2R].
\eeaa
Together with \eqref{U-lower-bdd}, one can apply the comparison principle to conclude that $U_n\leq \bar{U}_n$ in $[-2R,2R]$ for all large $n$. In particular, we have
\bea\label{Vn-to-0}
\sup_{\xi\in[-R,R]}|U_n(\xi)|\to 0\quad \mbox{as}\quad n\to\infty.
\eea
Thanks to \eqref{Vn-to-0} and the $C^{2+\alpha}$ bound of $V_n$, up to subtract a subsequence, we may assume that
$V_n\to V_{R}$ uniformly in $[-R,R]$ as $n\to\infty$,
where $V_{R}$ is defined in $[-R,R]$ and satisfies
$V_{R}(0)=1/2$, $V'_{R}\geq0$ in $[-R,R]$ and
\beaa
dV_R''+rV_R(1-V_R)=0, \quad \xi\in[-R,R].
\eeaa
Next, by standard elliptic estimates and taking $R\to\infty$, up to subtract a subsequence, we may assume that $V_R\to V_{\infty}$ locally uniformly in $\mathbb{R}$ as $n\to\infty$, where $U_{\infty}$ satisfies
\beaa
dV_{\infty}''+rV_{\infty}(1-V_{\infty})=0, \quad \xi\in\mathbb{R},\quad V_{\infty}(0)=1/2,\quad  V'_{\infty}\geq 0.
\eeaa
It is not hard to see that $V_{\infty}(-\infty)=0$ and $V_{\infty}(+\infty)=1$. Therefore, $V_{\infty}$ forms a traveling front with speed $c=0$,
which is impossible since such solutions exist only for $c\geq 2\sqrt{dr}$ (see \cite{KPP}). This contradiction shows the existence of $a_1$. The proof of Theorem \ref{th:conditions} is complete.
\end{proof}


\bigskip

\noindent{\bf Acknowledgement.}
The author is supported by the Japan Society for the Promotion of Science P-23314.
The author also would like to thank Professor Chang-hong Wu at National Yang Ming Chiao Tung University and Professor Quentin Griette at at Universit\'{e} Le Havre Normandie for valuable suggestions and discussions.

\end{document}